\documentclass[letterpaper,11pt,leqno]{amsart}

\usepackage{geometry}
\usepackage[latin1]{inputenc}

\usepackage{latexsym}
\usepackage[dvips]{graphicx}
\usepackage{amsthm} 
\usepackage{epsfig}
\usepackage{color}
\usepackage{amsmath}
\usepackage{amssymb}
\usepackage{pict2e}
\usepackage{hyperref}
\hypersetup{
    colorlinks,
    citecolor=black,
    filecolor=black,
    linkcolor=blue,
    urlcolor=blue
}
\usepackage{dsfont}
\usepackage{geometry}
\usepackage{enumerate}

\newtheorem{theorem}{Theorem}[section]

\theoremstyle{definition}
\newtheorem{lemma}[theorem]{Lemma}
\newtheorem{corollary}[theorem]{Corollary}
\newtheorem{proposition}[theorem]{Proposition}
\newtheorem{remark}[theorem]{Remark}
\newtheorem{example}[theorem]{Example}

\newcommand{\R}{\mathbb{R}}

\newcommand{\C}{\mathcal{C}}
\newcommand{\E}{\mathcal{E}}

\renewcommand{\H}{\mathcal{H}}

\renewcommand{\S}{\mathcal{S}}

\newcommand{\eps}{\varepsilon}

\newcommand{\proj}{\mbox{\rm proj}}

\newcommand{\co}{\mbox{\rm co}}

\newcommand{\argmin}[1][]{ \underset{#1}{\mbox{\rm argmin}} \ }

\newcommand{\Min}[1][]{ \underset{#1}{\mbox{\rm MIN}} \ }

\newcommand{\IMOG}{\mbox{{\rm (IMOG)}}}
\newcommand{\MOG}{\mbox{{\rm (MOG)}}}

\DeclareFontEncoding{FMS}{}{}
\DeclareFontSubstitution{FMS}{futm}{m}{n}
\DeclareFontEncoding{FMX}{}{}
\DeclareFontSubstitution{FMX}{futm}{m}{n}
\DeclareSymbolFont{fouriersymbols}{FMS}{futm}{m}{n}
\DeclareSymbolFont{fourierlargesymbols}{FMX}{futm}{m}{n}
\DeclareMathDelimiter{\VERT}{\mathord}{fouriersymbols}{152}{fourierlargesymbols}{147}

\DeclareFontFamily{U}{matha}{\hyphenchar\font45}
\DeclareFontShape{U}{matha}{m}{n}{
      <5> <6> <7> <8> <9> <10> gen * matha
      <10.95> matha10 <12> <14.4> <17.28> <20.74> <24.88> matha12
      }{}
\DeclareSymbolFont{matha}{U}{matha}{m}{n}

\DeclareMathSymbol{\leq}         {3}{matha}{"A4}
\DeclareMathSymbol{\geq}         {3}{matha}{"A5}
\DeclareMathSymbol{\nleq}        {3}{matha}{"A6}
\DeclareMathSymbol{\ngeq}        {3}{matha}{"A7}
\DeclareMathSymbol{\lneq}        {3}{matha}{"AC}
\DeclareMathSymbol{\gneq}        {3}{matha}{"AD}
\DeclareMathSymbol{\leqq}        {3}{matha}{"AE}
\DeclareMathSymbol{\geqq}        {3}{matha}{"AF}

\DeclareFontFamily{U}{mathb}{\hyphenchar\font45}
\DeclareFontShape{U}{mathb}{m}{n}{
      <5> <6> <7> <8> <9> <10> gen * mathb
      <10.95> mathb10 <12> <14.4> <17.28> <20.74> <24.88> mathb12
}{}
\DeclareSymbolFont{mathb}{U}{mathb}{m}{n}

\DeclareMathSymbol{\prec}        {3}{mathb}{"A0}
\DeclareMathSymbol{\preceq}      {3}{mathb}{"A8}
\DeclareMathSymbol{\precneq}     {3}{mathb}{"AC}
\DeclareMathSymbol{\llcurly}     {3}{mathb}{"CE}

\begin{document}

\title{Multiobjective optimization : an inertial dynamical approach to Pareto optima}

\author{H\'edy Attouch \& Guillaume Garrigos}

\address{Institut de Math\'ematiques et Mod\'elisation de Montpellier, UMR 5149 CNRS, Universit\'e de Montpellier, place Eug\`ene Bataillon, 34095 Montpellier cedex 5, France}

\email{hedy.attouch@univ-montp2.fr; guillaume.garrigos@gmail.com}

\vspace{0.5cm}


\maketitle

\markboth{H\'edy Attouch \& Guillaume Garrigos}
  {Multiobjective optimization: an inertial dynamic}

\maketitle

\if{\color{red}TODO : 
write introduction and first section with definition of Paretos, steepest descent direction etc..

rewrite asymptotic part and existence with general initial time $t_0$ instead of 0

talk about uniqueness. Can we get generic uniqueness in the space of smooth functions from $\R^n$ to $\R^q$? In Baire sense?

 existence in convex case
 
  examples with numerical simulations, and exploration of Paretos?
  
  talk about general differential inclusion
  
algorithm

studying what happen when $m \to 0$. 

Tikhonov regularization to get the projection of some initial point onto the set of Paretos?

}\fi

\textbf{Abstract.} We present some first results concerning a gradient-based dynamic approach to multi-objective optimization problems, involving inertial effects. We prove the existence of global solution trajectories for this second-order differential equation, and their convergence to weak Pareto points in the convex case. It is a first step towards the design of fast numerical methods for multi-objective optimization.

\section{Introduction}

We propose a first study of an inertial gradient-based dynamical system for multi-objective optimization. 
 In a Hilbert space setting $\mathcal H$, given  objective functions $f_i : \mathcal H \rightarrow \mathbb R$, \ $i=1,2,...,q$, which are  continuously differentiable, we consider the Inertial Multi-Objective Gradient system 
\begin{equation*}
{\rm (IMOG)} \ \  m \ddot u(t) + \gamma \dot u(t) 
+  \co \nabla f_i(u(t))^0 =0 .
\end{equation*}
It is a second-order in time differential equation, where  the mass $m$ and the viscous damping coefficient $\gamma$  are fixed positive parameters, and $\co \nabla f_i(u(t))^0$ is the element of minimal norm of the convex hull of the gradients of the objective functions at $u(t)$.
This dynamical system combines both aspects, inertial and multi-objective. Each of them has been the subject of active research, but to our knowledge the combination of both aspects has not been considered before.
Let us review some important facts concerning each of these aspects separately.

\smallskip

a) When neglecting the acceleration term, we  recover the first-order system
\begin{equation*}
{\rm (MOG)} \quad \dot u(t) 
+ \co \nabla f_i(u(t))^0 =0,
\end{equation*}
 which has been first considered by  Henry \cite{Hen}, Cornet \cite{Corn1}-\cite{Corn3} in economics, as a dynamical mechanism of ressource allocation. 
It was then developed independently as an optimization tool by Miglierina \cite{Mig04}, Brown and Smith \cite{BrS}, Attouch and Goudou \cite{AttGou14}. Its extension to the nonsmooth setting has been recently considered in Attouch, Garrigos and Goudou \cite{AttGarGou15}.
As a main property of (MOG), along its trajectories all the objective functions are decreasing. In the quasi-convex case, the trajectories converge  as $t \to + \infty$ to Pareto optima.\\
Various  first-order algorithms for multi-objective optimization can be considered as the time discretization  
 of this dynamic: let us refer to the seminal work of
 Fliege and Svaiter \cite{FliSva00},  followed by  \cite{BenFerOli12,DruIus04,DruRauSva14,DruSva05,
 MigMolRec08} (and the references therein).
Newton-based methods  have been considered by Drummond, Fliege and Svaiter \cite{DruFliSva09}, see also the multiobjective BFGS method of Povalej \cite{Pov14}.

\smallskip

b) When considering a single criteria $f$, we recover the so-called Heavy Ball with Friction dynamic 
\begin{equation*}
{\rm(HBF)} \ \ m \ddot u(t) + \gamma \dot u(t) + \nabla f(u(t))=0.
\end{equation*}
This system has a clear mechanical interpretation. Just  like a heavy ball sliding down the graph of $f$, due to the viscous friction effect, each trajectory tends to stabilize at a local minimum of $f$.
As an optimization tool, this system was first considered by  Polyack \cite{Poly},  Antipin \cite{Ant}, and Attouch-Goudou-Redont \cite{AttGouRed00}.
The convergence property of the trajectories has been proven in the two basic situations, in the case $f$ real analytic by  Haraux and Jendoubi \cite{HJ}, and in the convex case by Alvarez \cite{Al}.
The basic motivation for considering second-order in time systems is that, intuitively inertia provides fast methods.
The study of fast gradient-based methods for solving optimization problems is an active area of research.
While first-order methods are well-understood and relatively easy to implement, they generally suffer from slow convergence rates (linear). 
In contrast, second-order methods usually enjoy fast convergence properties, as super-linear or quadratic.

Generally speaking, there are two ways to incorporate  second-order information in the dynamics, or the algorithms. A second-order analysis in space (using the Hessian of the objective function, or some approximation),  leads to Newton-like methods. 
The other approach, which is our  main concern,  consists in using a second-order information in time, by introducing in the dynamics the second-order derivative of the trajectory. 
Such inertial methods are easier to implement than the Newton-like methods, but their analysis is quite  delicate. 
Following Nesterov and G\"uler seminal work, a very popular 
method is the  FISTA algorithm which has been developed by  Beck and Teboulle \cite{BT}, and which is an inertial version of the classical forward-backward algorithm. As a remarkable property of this algorithm, the  
 convergence rate of the values is $\mathcal O (\frac{1}{k^2})$.
Recently  Su, Boyd and Candes \cite{SBC} showed that this algorithm can be interpreted as a discrete version of the second-order differential equation, in the case $\alpha =3$,
 \begin{equation*}
  \ddot u(t) + \frac{\alpha}{t} \dot u(t) + \nabla f(u(t))=0.
\end{equation*}
 In the above equation, the viscous coefficient $\frac{\alpha}{t}$ tends to zero as $t\to +\infty$, which makes the inertia effect more effective asymptotically than with a fixed positive viscous coefficient, a key for fast methods (see Cabot, Engler and Gaddat \cite{CEG} for a general view on the asymptotic vanishing damping effect).
Convergence of the trajectories of the above system has been obtained in the case $\alpha >3$ by Attouch-Peypouquet-Redont \cite{APR1} (continuous dynamic) and Chambolle-Dossal \cite{CD} (discrete algorithmic case).

\smallskip

As a general rule, comparison of the corresponding continuous and discrete dynamics is of interest, since they usually share asymptotically the same qualitative and quantitative properties: convergence to a critical point of the objective function, with similar convergence rate.
For a rigorous approach of this comparison, see the works of Alvarez-Peypouquet \cite{AlvPey10,AlvPey11}, and Peypouquet-Sorin \cite{PeySor10}.
Moreover the continuous dynamic is often easier to treat mathematically than the corresponding algorithms, since we can use the flexibility of the differential and integral calculus. Quite often Lyapunov functions are first discovered in the continuous case, and then transposed to the algorithms.

Thus our program consists in studying the  Inertial Multi-Objective Gradient system, by combining the  technics which have been described above.
In section 2, we recall briefly some aspects of the multiobjective optimization, and in particular the first-order steepest descent method evoked above. Then we introduce our dynamic, namely the Inertial MultiObjective Gradient (IMOG) system.
In section 3, we investigate the existence of solution trajectories of (IMOG) in finite dimensions.
In section 4, we study the properties of the trajectories generated by (IMOG). Under a convexity assumption on the objective functions, we show that the bounded trajectories converge to weak Pareto points of the problem.
Of course, due to the effects of inertia, (IMOG) is not a descent dynamic, i.e. the values of the cost functions may not decrease over time.
But we show that, with an appropriate choice of the initial velocity, the cost values are improved along the trajectory relative to the starting point.

\section{Pareto Optimality and Multi-objective steepest descent direction}

Let $f_i : \H \longrightarrow \R$ ($i\in\{1,...,q\}$) be a finite family of real-valued functions. We suppose in this paper that they are continuously differentiable, with gradients being Lipschitz continuous on bounded sets. 
Note $F : \H \longrightarrow \R^q$ the vector-valued function defined by $F(u):=(f_i(u))_{i \in \{1,...,q\}}$, and consider the associated vector optimization problem
\begin{equation*}
{\rm (P)} \quad \Min[u \in \H] F(u).
\end{equation*}
Let us precise the notions of solutions we consider for $(P)$. 

Take the canonical order $\preceq$ on $\R^q$ defined by
\begin{equation}
a \preceq b \Leftrightarrow \ \forall i \in \{1,...,q\}, \ a_i \leq b_i,
\end{equation}
which induces a strict order $a \precneq b \Leftrightarrow a \preceq b$ and $a \neq b$. 
Then, we say that $u \in \H$ is a \textit{Pareto efficient point} (or \textit{Pareto optimum)} of $(P)$ whenever the sublevel set $\{ v \in \H \ | \ F(v) \precneq F(u) \}$ is empty. 
In other words, Pareto optimum are points having the property that none of the objective functions can be improved in value, without degrading some of the other objective values.
\noindent We can also equip $\R^q$ with a weaker strict order $\prec$, defined by
\begin{equation}
a \prec b \Leftrightarrow \ \forall i \in \{1,...,q\}, \ a_i < b_i.
\end{equation}
To this weaker strict order corresponds a weaker notion of Pareto efficiency : we say that $u \in \H$ is a \textit{weak Pareto efficient point} (or \textit{weak Pareto optimum)} of $(P)$,  if $\{ v \in \H \ | \ F(v) \prec F(u) \}$ is empty.
It is clear that any Pareto optimum is in particular a weak Pareto optimum. 
Note that in the case of a single objective function $f$ (i.e. $q=1$), Pareto and weak Pareto optima  coincide with the notion of global minimizer of $f$. 

Still in this mono-criteria case, a known necessary condition for $u\in \H$ to be a minimizer is $\nabla f(u)=0$. A generalization of this Fermat's rule holds for Pareto optima :

\begin{proposition}{(Fermat's rule)}\label{P:Fermat}
If $u$ is a weak Pareto point of $(P)$, then $0 \in \co \nabla f_i(u)$.
\end{proposition}

\noindent By analogy with the mono-criteria case, we say that $u$ is a \textit{critical Pareto} point whenever $0 \in \co \nabla f_i(u)$.
This notion has been considered by Smale in \cite{Sma73}, Cornet in \cite{Corn1}, \ see \cite{BaoMor12} for recent account of this notion, and various extensions of it. 
As we can expect, this first-order necessary optimality condition for local multi-objective optimization becomes  sufficient in the convex setting (see for instance \cite[Lemma 1.3]{AttGarGou15}) :

\begin{proposition}\label{P:Fermat CNS}
If each objective function $(f_i)_{i\in\{1,...,q\}}$ is convex, then Pareto critical points coincide with  weak Pareto points. If the functions are strictly convex, then the same holds for Pareto optima.
\end{proposition}

\noindent These last properties justifies the search for critical Pareto points, just as critical points are looked for in the single objective case.


~~

We introduce now  the \textit{steepest descent} vector field :
\begin{eqnarray}\label{E:steepest descent}
s : \H & \longrightarrow &\H \\
  u& \longmapsto & s(u):=-\co \nabla f_i(u)^0 \nonumber
\end{eqnarray}
where $\co \nabla f_i(u)^0$ denotes the element of minimal norm of the convex compact set $\co \nabla f_i(u)$. The vector $s(u)$ is called the \textit{multi-objective steepest descent direction} at $u$, and simply reduces to $- \nabla f(u)$ if $q=1$. 
It enjoys the following nice properties, which extends known facts about $-\nabla f(u)$ in the mono-criteria case :
\begin{enumerate}
	\item $s(u)=0$ if and only if $u$ is Pareto critical.
	\item $s(u)$ is a \textit{common} descent direction at $u$ for all the objective functions. More exactly, 
	\begin{equation}\label{E:common descent property}
	\forall i=\{1,...,q\}, \ \langle \nabla f_i (u) , s(u) \rangle \leq - \Vert s(u) \Vert^2.
	\end{equation}
	\item It is the \textit{steepest} common descent direction, in the sense that
	\begin{equation}\label{E:steepest descent property}
	\frac{s(u)}{\Vert s(u) \Vert} = \argmin[d\in \H]  \max\limits_{i\in\{1,...,q\}} \ \langle \nabla f_i (u), d \rangle \ \text{ whenever } s(u) \neq 0. 
	\end{equation}
\end{enumerate}
Item i) is immediate since $s(u)$ is defined as the element of minimal norm of $-\co \nabla f_i (u)$. 
Item ii) comes immediately once seeing that $s(u)$ is the projection of the origin onto  $-\co \nabla f_i (u)$. 
Item iii) makes use of duality arguments, see for instance \cite[Proposition 3.1]{Corn1} or \cite[Theorem 1.8]{AttGarGou15} for more details on the proof.

Because of its properties, it seems natural to consider the dynamic governed by the steepest descent vector field $s : \H \longrightarrow \H$, namely the \textit{MultiObjective Gradient} system :
\begin{equation}
\MOG \ \ \dot u(t) - s(u(t)) =0 , \ t \in [0,+\infty[.
\end{equation}
Indeed, at each $t$, $\dot u(t)$ is a common descent direction for the objective functions, so they should all decrease along the trajectories. 
Also, equilibrium points of the dynamic are exactly critical Pareto points.
Note that the idea of constructing dynamic having these properties goes back to Smale in \cite{Sma73}.
This dynamic has been recently studied in \cite{AttGou14,AttGarGou15}, where the authors prove the \textit{cooperative} nature of this dynamic, that is the \textit{common} decrease of the objective functions along the trajectories.
Moreover, the trajectories are proved to weakly converge  to Pareto critical points, in the convex case.
This steepest descent dynamic can be seen as the continuous version of the gradient method introduced by Fliege and Svaiter in \cite{FliSva00},  which share the same qualitative and asymptotic behavior with $\MOG$.

~~

Our purpose is to consider a modified version of this steepest descent dynamic, by introducing inertial effects. 
Take $m,\gamma >0$, and consider the following \textit{Inertial MultiObjective Gradient}  dynamic:
\begin{equation}
{\rm (IMOG) }\ \ m \ddot u(t) + \gamma \dot u(t) - s(u(t)) =0, \ t \in [0,+\infty[.
\end{equation}
From a physical point of view, the parameter $m$ can be interpreted as the mass of the physical point $u(t)$, on which acts the sum of two forces : the friction $-\gamma \dot u(t)$ and the vector field $s(u(t))$. 
One can see, at least formally, that we recover the first-order steepest descent dynamic (SD) by letting the mass $m$  go to zero.
Moreover, it reduces to the classical Heavy Ball with Friction (see \cite{AttGouRed00}) when $q=1$, that is :
\begin{equation}
{\rm (HBF)} \ \ m \ddot u(t) + \gamma \dot u(t) + \nabla f(u(t)) = 0,  \ t \in [0,+\infty[.
\end{equation}
As we will see in section 4,  (IMOG) dynamic shares similar properties with  (HBF).

\if{\color{red}TO DO : Talk about $s$ used to design algorithms and continuous dynamic. Insist on the fact that the associated first-order dynamics are cooperative. Of course, because of the inertial effects, it cannot be hoped to have a common descent for IMOG. But by choosing appropriately the initial velocity, it can be shown that the values along the trajectory are improved in comparison with the initial point, see Proposition \ref{P:boundedness sufficient condition}. This property does not holds if $s(u)$ is replaced by the gradient of a fixed convex combination. Take as a counter example the distances to two points, take any non-trivial scalarization, and start from one of the two points. Then for any initial velocity, the trajectory is escapes from the initial point and strictly increases the distance to this point.
}\fi

\section{Existence of trajectories for (IMOG)}

\subsection{Existence of trajectories}

In this section, we question the existence of solutions for the Cauchy problem associated to (IMOG). 
Let $t_0 \in \R$, $T \in ]t_0,+\infty]$, and $(u_0,\dot u_0) \in \H^2$. 
We say that $u:[t_0,T [ \longrightarrow \H$ is a solution of {\rm (IMOG)} if $u$ is continuous on $[t_0,T [$, of class $C^2$ on $]t_0,T[$, and satisfies 
\if{\begin{equation*}\label{D:IMOG}
\text{(IMOG)} \ \  \ \ m \ddot u(t) + \gamma \dot u(t) + \proj_{\co \{\nabla f_i(u(t))\}}(- \alpha m \ddot u(t))=0 \ \text{ with } u(t_0)=u_0, \dot u (t_0)=\dot u_0.
\end{equation*}
We restrict our study to the case $\alpha =0$, since in this case the dynamic is governed by the autonomous vector field $s : u \in \H \mapsto \proj_{\co  \nabla f_i (u)} (0)$, making the analysis simpler. 
That is, given $(u_0,\dot u_0) \in \H^2$ and $I$ an open interval of $\R$ containing $t_0$, we look for a solution of :
$$\begin{array}{l|l}
(IMOG) \ \ & m \ddot u(t) = -\gamma \dot u(t) + s(u(t)) \ \text{ for all } t \in ]t_0,T[, \\
  & u(t_0)=u_0, \ \dot u(t_0)=\dot u_0 .
\end{array}$$
}\fi
$$\begin{array}{l|l}
\ \ & m \ddot u(t) = -\gamma \dot u(t) + s(u(t)) \ \text{ for all } t \in ]t_0,T[, \\
{\rm (IMOG)} \ \ & \\
  & u(t_0)=u_0, \ \dot u(t_0)=\dot u_0 .
\end{array}$$
In view to apply an existence theorem for the dynamical system ${\rm (IMOG)}$, the key point is the regularity of the steepest descent vector field $s$. We recall the following result from \cite{AttGou14} :

\begin{proposition}\label{P:Holder continuity}
Recalling that the gradients $\nabla f_i : \H \rightarrow \H$ are Lipschitz continuous on bounded sets, we have that $s$ is $\frac{1}{2}$-Hölder continuous on bounded sets. 
\end{proposition}

\noindent This result is nearly optimal since there exists simple situations for which $s$ is not a locally Lipschitz continuous vector field (see \cite[Example 2]{AttGarGou15}). Hence, there is no hope to apply Cauchy-Lipschitz's theorem to get existence and uniqueness of the trajectories. We will use instead Peano's existence result (see for instance \cite[Theorem 2.8]{AyeDomLop}) :

\begin{theorem}\label{T:Peano}(Peano)
Let $\phi : \R^n \longrightarrow \R^n$ be continuous. Then, for all $x_0 \in \H$ and $t_0 \in \R$, there exists some $T >0$ and $x : [t_0,t_0 + T[ \longrightarrow \H$ of class $C^1$, such that
\begin{equation}\label{E:Peano}
\dot x(t) = \phi(x(t)) \ \text{ for all } t\in [t_0,t_0+T[,  \text{ with } x(t_0)=x_0.
\end{equation}
\end{theorem}

As one knows, Peano's result asks less regularity but applies only in finite dimension. Moreover, contrary to the Cauchy-Lipschitz theorem, uniqueness is not guaranteed here (it will be discussed later). The ingredients are now all gathered to get a first local existence result :

\begin{proposition}\label{P:local existence}
Suppose that $\H$ has finite dimension. For all $t_0 \in \R$, for all $(u_0,\dot u_0) \in \H \times \H$, there exists some $T >0$ and $u : [t_0,t_0 + T[ \longrightarrow \H$ of class $C^2$, such that
\begin{equation}\label{E:local existence1}
 m \ddot u(t) = -\gamma \dot u(t) + s(u(t)) \ \text{ for all } t\in [t_0,t_0+T[,  \text{ with } u(t_0)=u_0, \dot u(t_0)=\dot u_0.
\end{equation}
\end{proposition}

\begin{proof}
We just need to apply a change of variables in $\IMOG$ to get a first-order ODE. Let $\phi : \H^2 \longrightarrow \H^2$ be defined by 
\begin{equation}\label{le1}
\phi(u,v):=\frac{1}{m}(mv, - \gamma v + s(u)).
\end{equation}
Clearly, from Proposition \ref{P:Holder continuity}, $\phi$ is continuous on $\H^2$. We can then apply Peano's Theorem at $t_0$ and $x_0:=(u_0,\dot u_0)$, to get some $x : [t_0,t_0+T[ \longrightarrow \H^2$ of class $C^1$ such that (\ref{E:Peano}) holds. If we note $x(t)=(u(t),v(t)) \in \H\times \H$, (\ref{E:Peano}) can be rewritten as :
\begin{equation}\label{le2}
\dot{u}(t) = v(t), m\dot v(t) = - \gamma v(t) + s(u(t)) \ \text{ for all } t\in [t_0,t_0+T[,  \text{ with } u(t_0)=u_0, v(t_0)=\dot u_0.
\end{equation}
Since $x$ is of class $C^1$, we deduce that it is also the case for $u$ and $v$. But from $\dot{u}(t) = v(t)$, we can see that $u$ is of class $C^2$ and satisfies (\ref{E:local existence1}).
\end{proof}

\begin{remark}
For this result we use Theorem \ref{T:Peano}, which asks the space to be finite dimensional. 
In fact, Peano's theorem can be stated in the Banach space setting, if one asks the vector field involved to be compact.  
We recall that an application $\phi : \H \longrightarrow \H$ is said to be compact whenever it is continuous and maps bounded sets to relatively compact sets.
Observe that if the gradients $\nabla f_i$ are all compact, then $s$ is also compact.
Hence, one might want to apply Peano's result in this context. 
Nevertheless, by reducing (IMOG) to a first-order ODE, we do not deal directly with $s$ but with $(u,v) \mapsto \frac{1}{m}(mv, - \gamma v + s(u))$. 
And it can be easily proved that if $s$ is compact,  then $v \mapsto v$ is also compact, which would mean that $\H$ has finite dimension.
\end{remark}

We can now state our main existence result. To get a global solution on $[0, + \infty[$, we do a stronger hypothesis on the gradients.

\begin{theorem}\label{T:global existence}
Suppose that $\H$ has finite dimension, and that the gradients $\nabla f_i$ are globally Lipschitz continuous. 
Then, for all $t_0 \in \R$, $(u_0,\dot u_0) \in \H \times \H$, there exists  $u : [t_0,+\infty[ \longrightarrow \H$ of class $C^2$, such that
\begin{equation}\label{E:local existence}
 m \ddot u(t) = -\gamma \dot u(t) + s(u(t)) \ \text{ for all } t\in [t_0,+\infty[,  \text{ with } u(t_0)=u_0, \dot u(t_0)=\dot u_0.
\end{equation}
\end{theorem}

\begin{proof}[Proof of Theorem \ref{T:global existence}]
Proposition \ref{P:local existence} provides us a local solution and, using Zorn's lemma, we can suppose that it is a maximal solution $u : [t_0,T[ \longrightarrow \H$, with $T \in [t_0,+\infty]$. 
The whole point is to prove that $T = +\infty$. 
For this, we argue by contradiction by supposing that $T < +  \infty$. 
We will show that the solution does not blow up in finite time, and extend it at $T$ to obtain a contradiction.

Using the fact that the gradients are globally Lipschitz continuous, we can derive the following global growth property for $s$ :
\begin{equation}\label{ge1}
\exists c>0 \text{ s.t. } \forall u\in \H, \ \Vert s(u) \Vert \leq c ( 1 + \Vert u \Vert ).
\end{equation}
Indeed, for all $u \in \H$, there exists $\theta(u) \in \S^q$ such that
\begin{equation}\label{ge2}
\Vert s(u) \Vert = \Vert {\sum\limits_{i=1}^{q} \theta_i(u) \nabla f_i(u)} \Vert \leq \sum\limits_{i=1}^{q} \theta_i(u) \Vert \nabla f_i(u) - \nabla f_i(0) \Vert + \Vert \nabla f_i(0) \Vert.
\end{equation}
Then (\ref{ge1}) holds by taking for instance $c_1= \max\limits_{i\in\{1,...,q\}}  \{ \Vert \nabla f_i(0) \Vert, Lip(\nabla f_i;\H) \}$.

From this growth condition, we will obtain some energy estimates on the trajectory. Let us show that $\dot u$ and $\ddot u$ lie in $L^\infty(t_0,T;\H)$. 
For this, we consider as before
\begin{equation}\label{ge3}
\phi : \H^2 \longrightarrow \H^2, \ (u,v)\mapsto \phi(u,v)=\frac{1}{m}(mv, - \gamma v + s(u)).
\end{equation}
By defining $x(t):=(u(t),\dot u(t))$ for all $t \in [t_0,T[$, we see that $\dot x(t) = \phi(x(t))$ on $[t_0,T[$.
\if{Consider now some $\bar x \in \H^2$ such that for all $t \in ]t_0,T[$, $x(t) \neq \bar x$. The existence of such a point is guaranteed for instance by Sard's theorem, since $x(\cdot)$ is of class $C^1$ and $\H^2$ is at least of dimension $2$. }\fi
Define  $h(t):=\Vert x(t) - x(t_0) \Vert$ on $[t_0,T[$, which is continuous on $ [t_0,T[$. 
Equip $\H^2$ with the scalar product inherited from $\H$, and note that $h^2$ is derivable on $[t_0,T[$, so we can write for all $t \in [t_0,T[$ :
\begin{equation}\label{ge4}
\frac{d}{dt} \frac{1}{2} h^2(t) = \langle \dot x(t) , x(t) - x(t_0) \rangle = \langle \phi(x(t)), x(t) - x(t_0) \rangle \leq \Vert \phi(x(t)) \Vert h(t).
\end{equation}
From the growth condition (\ref{ge1}) we deduce an upper bound for $\Vert \phi(x(t))\Vert$. Indeed,  for all $x=(u,v) \in \H^2$,
\begin{eqnarray*}
\Vert \phi(x) \Vert  \leq & \frac{1}{m}( \Vert m v \Vert + \Vert s(u) - \gamma v \Vert ) & \text{ using the equivalence between $\ell^1$ and $\ell^2$ norms} \\
\leq &(1 + \frac{\gamma}{m}) \Vert v \Vert +  \frac{c}{m}(1 + \Vert u \Vert) & \text{ using the triangle inequality with (\ref{ge1})} \\
\leq & c_2 (1+ \Vert x \Vert ) & \text{ with } c_2:= \sqrt{2} \max \{ \frac{c}{m}; 1 + \frac{\gamma}{m} \}.
\end{eqnarray*}
Using the triangle inequality with $c_3:=c_2(1+\Vert x(t_0) \Vert)$, it follows for all $t \in [t_0,T[$ that
\begin{equation}\label{ge5}
\Vert \phi(x(t))\Vert \leq c_3(1 + h(t)).
\end{equation}
Combining (\ref{ge4}) and (\ref{ge5}), we obtain
\begin{equation}\label{ge6}
\frac{d}{dt} \frac{1}{2} h^2(t) \leq c_3 h(t) (1 + h(t)) \text{ for all } t \in [t_0,T[.
\end{equation}

We will now conclude by using a Gronwall-type argument. Consider an arbitrary $\eps \in ]0,T-t_0[$. 
After integration of (\ref{ge6}) on $[t_0,T-\eps]$, and using $h(t_0)=0$, we obtain
{
\begin{equation}\label{ge6.1}
\frac{1}{2}h^2(t) \leq  \int_{t_0}^t c_3(1+h(s)) h(s) \ ds \text{ for all } t \in [t_0,T-\eps] .
\end{equation}
Since $h$ is continuous on $[0,T-\eps]$, the function $g:t \in [0,T- \eps] \mapsto c_3 (1+h(t))$ is in $L^1([0,T-\eps],\R)$. 
Hence we can apply Lemma \ref{L:Gronwall 2} (we left it in the Appendix) to obtain
\begin{equation}\label{ge6.2}
h(t) \leq \int_{t_0}^t c_3 (1+h(s)) \ ds \text{ for all } t \in [t_0,T-\eps].
\end{equation}
We easily obtain from (\ref{ge6.2}) and $T < + \infty$ that
\begin{equation}\label{ge6.3}
h(t) \leq  c_3 T + c_3\int_{t_0}^t h(s) \ ds \text{ for all } t \in [t_0,T-\eps],
\end{equation}
so we can use the Gronwall-Bellman's Lemma (see Lemma \ref{L:Gronwall-Bellman} in the Appendix), and obtain :
\begin{equation}\label{ge6.4}
h(t) \leq  c_3 T   e^{c_3 t} \leq c_3 T  e^{c_3 T} \text{ for all } t \in [t_0,T-\eps].
\end{equation}
}
Since the upper bound in (\ref{ge6.4}) is independent of $\eps$ and $t$, we deduce that $h \in L^\infty([0,T],\R)$.

\if{Since $\frac{d}{dt} \frac{1}{2} h^2(t) = h'(t) h(t)$, it follows after dividing (\ref{ge6}) by $h(t)>0$ that 
\begin{equation}\label{ge7}
h'(t) \leq c_3(1+h(t)).
\end{equation}
Applying then Gronwall's Lemma to $t \mapsto (1+ h(t))$, we conclude from $T<+\infty$ that $h \in L^\infty(0,T;\R)$. }\fi

From the definition of $h$, we obtain that $u$ and $\dot u$ lie in $L^\infty (0,T;\H)$. Moreover, using the growth condition (\ref{ge1}), we see that $s\circ u \in L^\infty(0,T;\H)$, so $\ddot{u}(t) = \frac{1}{m}(s(u(t)) - \gamma \dot u(t))$ lies also in $L^\infty(0,T;\H)$.
Now, since $T$ is supposed finite, we can say that $L^\infty(0,T;\H) \subset L^1(0,T;\H)$, so $u$ can be extended continuously at $T$ by $u(T):=u(0) + \displaystyle \int_0^T \dot u(t) dt$, and we can do the same for $\dot u$. Hence, we can apply Proposition \ref{P:local existence} at $t_0=T$ with $(u_0,\dot u_0)=(u(T),\dot u(T))$ to extend the solution $u(\cdot)$, which contradicts its maximality.
\end{proof}
\if{\color{red}
\begin{remark}
In order to integrate inequation (\ref{ge7}), another method would consist in using the Gronwall-Bellman  lemma, 
see Brezis \cite[Lemme A.5.]{Br}.
\end{remark}
}\fi

As already observed in \cite{AttGarGou15}, the steepest descent vector field governing the dynamic is not Lipschitz, neither monotone (even in the convex setting).
So we cannot use methods from monotone operator theory, and the question of uniqueness of the trajectories remains open in the general context.
Nevertheless, under some assumptions, we still can ensure the uniqueness.

\begin{proposition}\label{P:uniqueness}
Let $u$ be a trajectory solution of the Cauchy problem (\ref{E:local existence}). Suppose that for all $t \in [t_0,+\infty[$, $(\nabla f_i (u(t)))_{i=1,...,q}$ are linearly independent vectors.
Then $u$ is the unique solution to (\ref{E:local existence}).
When $q=2$, the same conclusion holds, just assuming that  $\nabla f_1(u(t)) \neq \nabla f_2(u(t)) $.
\end{proposition}

\begin{proof}
It is proved in \cite[Proposition 3.4]{AttGarGou15}  that under these hypotheses, {for all $t \in [t_0,+\infty[$, the steepest descent vector field is locally Lipschitz  in the neighbourhood of $u(t)$}.
Hence, it suffices to apply the Cauchy-Lipschitz theorem instead of Peano's to derive the uniqueness of $u$.
\end{proof}

\subsection{Examples}

\begin{example}
Take the quadratic functions $f_1(x,y)= \frac{1}{2}(x+1)^2    + \frac{1}{2}y^2 $  and  $f_2(x,y)= \frac{1}{2}(x-1)^2    + \frac{1}{2}y^2 $. The corresponding Pareto set is $[-1, +1] \times \left\{0\right\}$ and the steepest descent vector field is given by  :

\begin{equation}
s(x,y)=
\begin{cases}
 -(x-1,y)     \  \ \mbox{if } \  x>1,\\
 -(0,y) \  \quad \quad \mbox{if } \  -1 \leq x \leq 1,\\
 -(x+1,y)     \ \  \mbox{if } \   x<-1 .
\end{cases}
\end{equation}

\noindent Figure 1 shows some trajectories of the (IMOG) dynamic,  with the steepest descent vector field plotted in background. We used the following parameters : $m=1$, and $(u_0,\dot u_0)$ are taken randomly.
Here the trajectories are computed exactly, since in this simple example (IMOG) can be solved explicitely. 
For each trajectory, the initial point is indicated by the symbol $\times$, and the limit point by $\oplus$.
We can observe the following :  the trajectories all converge to a Pareto point, the dynamic is clearly  not a descent method, and can be highly oscillating whenever the friction parameter is too close to zero.

\end{example}
{
\begin{center}
\begin{figure}[h!]
   \begin{minipage}[l]{.46\linewidth}
      \includegraphics[scale=0.4]{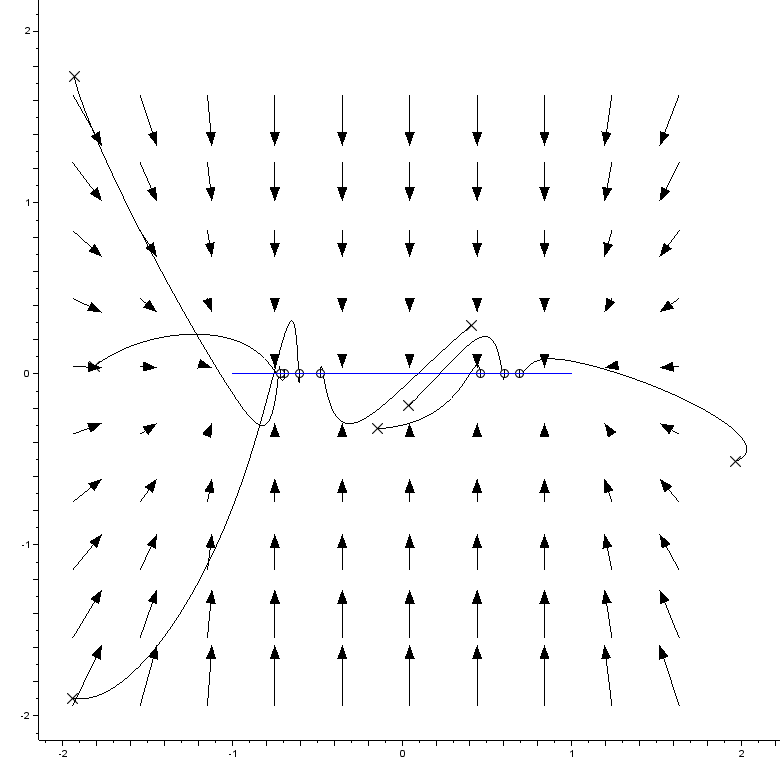}
   \end{minipage} \hfill
   \begin{minipage}[l]{.46\linewidth}
      \includegraphics[scale=0.4]{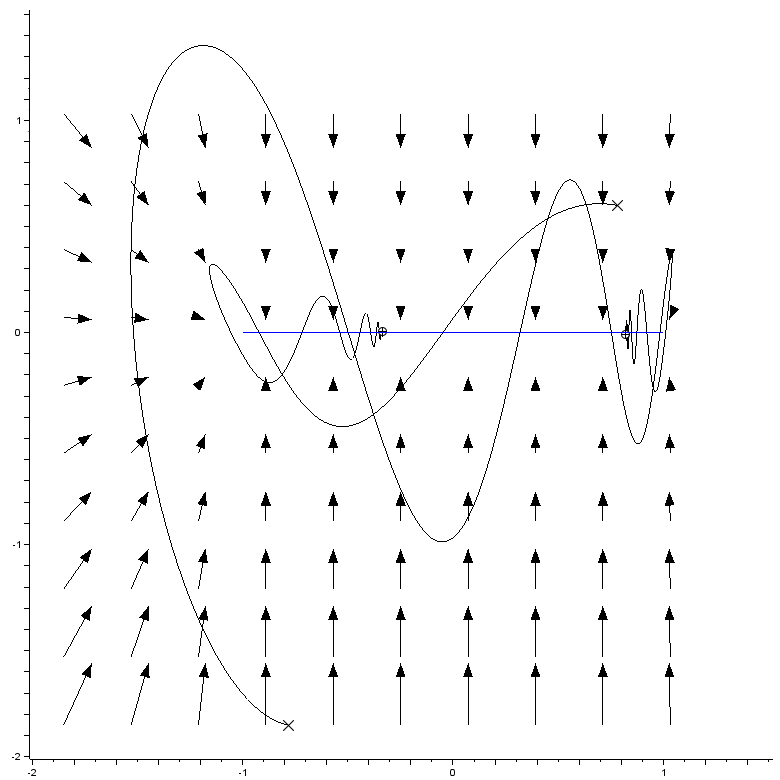}
   \end{minipage}
Figure 1. Friction parameter $\gamma=1$ (left) and $\gamma=0.1$ (right).
\end{figure}
\end{center}
}

\begin{example}
Let $f_1 (x,y) =\frac{1}{2}(x^2 + y^2)$ and   $f_2 (x,y)= x$. The corresponding Pareto set is $]-\infty, 0] \times \{0\}$, plotted in blue in Figure 2. Once computed, we see that the steepest descent vector field is defined according to three areas of the plane (these areas are delimited by red lines in blue in Figure 2) :  
\begin{equation}
s(x,y)=
\begin{cases}
 -(1,0)   \quad   \quad   \quad   \quad   \quad \quad  \ \  \quad \mbox{if } \  x\geq 1,\\
 -(x,y) \    \quad   \quad   \quad   \quad    \quad   \  \quad  \quad \mbox{if } \  (x-\frac{1}{2})^2 + y^2 \leq \frac{1}{4},\\
 \frac{-1}{(x-1)^2+y^2} (y^2,y(1-x))      \ \  \  \mbox{else} .
\end{cases}
\end{equation}
In this case we plotted the trajectories using an explicit discretization in time of (IMOG) :
\begin{eqnarray*}
&m\dfrac{u_{n+1}-2u_n + u_{n-1}}{t} + \gamma \dfrac{u_{n+1}-u_n}{t} + s(u_n) =0 \\
\Leftrightarrow & u_{n+1} = u_n + \frac{m}{m+t\gamma}(u_n - u_{n-1}) - \frac{t^2}{m+t\gamma}s(u_n).
\end{eqnarray*}
We took $m=1,\gamma =1$ and $t=0.05$, and again the initial point for each trajectory is indicated by the symbol $\times$, and the limit point by $\oplus$.

{
\begin{center}
\begin{figure}[h!]
		\includegraphics[scale=0.39]{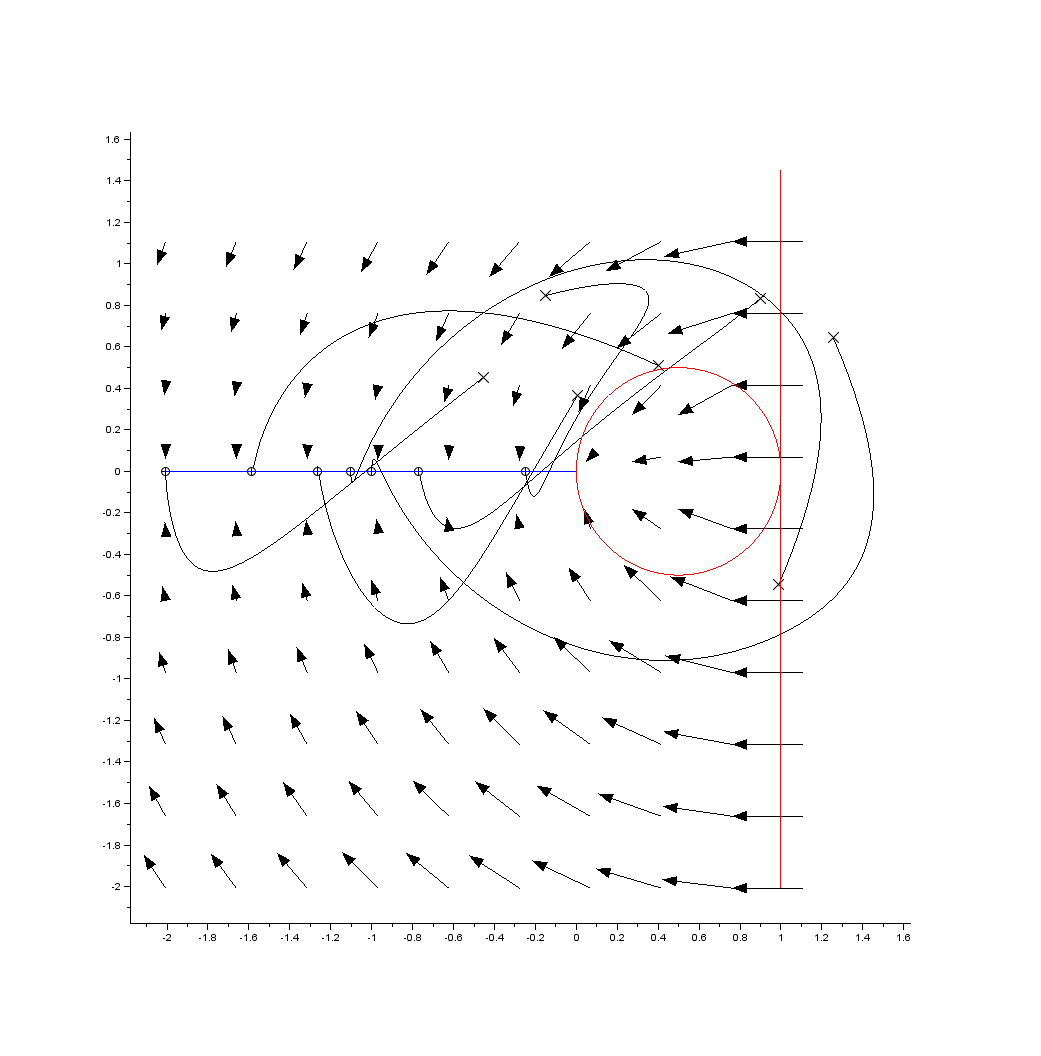}
		
Figure 2
\end{figure}
\end{center}
}

\end{example}

\if{\color{red}Pour terminer cette section, il serait intéressant de montrer l'existence dans le cas convexe sans global Lipschitz continuité des gradients (par régularisation Moreau-Yosida). Pour ces deux résultats, le point clé est qu'il faut prouver des estimations sur $u$. Or si on regarde attentivement les preuves des Proposition \ref{P:dissipative property} et \ref{P:energy estimations}, on voit qu'il est difficile de se passer de l'hypothèse que $u$ est bornée. 

Il faut aussi parler de l'unicité.
}\fi

\section{Properties of the solutions of (IMOG)}

For a given function $\phi:\H \longrightarrow \H$ and a nonempty subset $A \subset \H$, note $Lip(\phi;A)$ the best Lipschitz constant of $\phi$ over $A$, that is $Lip(\phi;A):= \sup\limits_{x\neq y \in A} \frac{\Vert \phi(x) - \phi(y) \Vert}{\Vert x - y \Vert}$. We say that $\phi$ is Lipschitz over $A$ whenever $Lip(\phi;A) < + \infty$.

\subsection{A dissipative system}

We start  our study of the dynamic by showing that it is a dissipative system. But before, we need the following chain rule:

\begin{lemma}(Chain rule)\label{L:chain rule second order}
Let $\phi : \H \longrightarrow \R$ and $u : I \longrightarrow \H$, where $I$ is a non-empty open subset of $\R$. 
Suppose that $\phi$ and $u$ are of class $C^{1,1}$ on $I$, and that $Lip(\nabla \phi;u(I)) <+\infty$.
Then for a.e. $t \in I$,
\begin{equation}\label{E:chain rule second order}
\frac{d^2}{dt^2} (\phi\circ u)(t) \leq Lip(\nabla \phi;u(I)) \Vert \dot u(t) \Vert^2 + \langle \nabla \phi(u(t)),\ddot u(t) \rangle.
\end{equation}
\end{lemma}

\begin{proof}
By hypothesis, $\dot u$ and $\nabla \phi \circ u$ are locally Lipschitz continuous, hence differentiable almost everywhere.
So, from $\frac{d}{dt} (\phi \circ u) (t) = \langle \nabla \phi \circ u (t) , \dot u(t) \rangle$, we have for a.e. $t \in I$ 
\begin{equation}
\frac{d^2}{dt^2} (\phi \circ u) (t) = \langle \frac{d}{dt} ( \nabla \phi \circ u) (t) , \dot u(t) \rangle + \langle \nabla \phi \circ u (t) , \ddot u(t) \rangle.
\end{equation}
Moreover, we have in the second member (using the Cauchy-Schwarz inequality and the Lipschitz property of $\nabla\phi$) :
\begin{eqnarray*}
\langle \frac{d}{dt} ( \nabla \phi \circ u) (t) , \dot u(t) \rangle  = &\lim\limits_{h \to 0} \frac{1}{h} \langle  \nabla \phi \circ u (t+h) -  \nabla \phi \circ u (t) , \dot u(t) \rangle \\
  \leq & \lim\limits_{h \to 0} \frac{1}{\vert h \vert} \Vert  \nabla \phi \circ u (t+h) -  \nabla \phi \circ u (t) \Vert \Vert \dot u(t) \Vert  \\
	\leq &   \lim\limits_{h \to 0} L \frac{1}{\vert h \vert} \Vert u(t+h) - u(t) \Vert \Vert \dot u(t) \Vert = L \Vert \dot u(t) \Vert^2 
\end{eqnarray*}
where $L:=Lip(\nabla \phi;u(I))$.

\end{proof}

Let us prove now the dissipativity of our dynamic :

\begin{proposition}(Dissipative property)\label{P:dissipative property} Let $u : [t_0,T[ \longrightarrow \H$ be a solution of $\IMOG$. For all $i\in\{1,...,q\}$, define for all $t \in [t_0,T[$ :
\begin{equation}\label{D:energy of the system}
\E_i(t):=(f_i \circ u)(t) + \frac{m}{\gamma} (f_i \circ u)'(t) +m \Vert \dot u(t) \Vert^2.
\end{equation}
Then, for a.e. $t\in[t_0,T[$, if $L_i:=Lip(\nabla f_i ; u([t_0,T[))<+\infty$, we have
\begin{equation}\label{E:dissipative property}
	\E_i'(t) \leq - \frac{m^2}{\gamma} \Vert \ddot u(t) \Vert^2 - \frac{1}{\gamma}\left( \gamma^2 - mL_i  \right) \Vert \dot u(t) \Vert^2
\end{equation}
\end{proposition}

\begin{proof}
The dissipative property is a direct consequence of the variational characterisation of the projection of $0$ over $\co \{ \nabla f_i(u(t))\}$ in (IMOG). 
Indeed, for a.e. $t \in [t_0,T[$, we have $-m \ddot u(t) - \gamma \dot u(t) = \proj_{\co \{\nabla f_i(u(t))\}}(0 )$.
It follows that
\begin{equation}\label{E:dp1}
 \langle  m \ddot u(t) + \gamma \dot u(t), \nabla f_i(u(t)) +  m \ddot u(t) + \gamma \dot u(t) \rangle \leq 0,
\end{equation}
which is equivalent, after distributing the terms and dividing by $\gamma$,  to
\begin{equation}\label{E:dp2}
\frac{m}{\gamma} \langle \nabla f_i(u(t)), \ddot u(t) \rangle 
+  \frac{d}{dt} [(f_i\circ u) + m   \Vert \dot u \Vert^2 ](t) 
\leq - \frac{m^2}{\gamma} \Vert \ddot u(t) \Vert^2 - \gamma \Vert \dot u(t) \Vert^2.
\end{equation}
Use now Lemma \ref{L:chain rule second order} with $ Lip(\nabla f_i;u([t_0,T[)) < + \infty$ to obtain
\begin{equation*}
\frac{m}{\gamma} \left( \frac{d^2}{dt^2} (f_i \circ u)(t) - L_i \Vert \dot u(t) \Vert^2 \right) + \frac{d}{dt} [(f_i\circ u) + m   \Vert \dot u \Vert^2 ](t) 
\leq - \frac{m^2}{\gamma} \Vert \ddot u(t) \Vert^2 - \gamma \Vert \dot u(t) \Vert^2,
\end{equation*}
which ends the proof.
\end{proof}

Proposition \ref{P:dissipative property} suggests that we need an hypothesis on the parameters to ensure the dissipative property :
\begin{eqnarray*}
\text{\rm (HP)}_i & \gamma^2 > mL_i \text{ where } L_i:= \ Lip(\nabla f_i;u([t_0,T[)). 
\end{eqnarray*}

\if{Let us make some observations on the hypotheses made here.
Hypothesis ii) is obviously satisfied if the objective functions are bounded from below. 
Because of i), it is also satisfied whenever the objective function are bounded on bounded sets, or weakly lower semi-continuous. This last assumption holds if $H$ finite dimensional or if the objective functions are convex.
Hypothesis iii) is more striking, since it asks the friction parameter $\gamma$ to be big enough (or, from an other point of view, it asks $\alpha$ to be closed enough to $1$). }\fi

\noindent If $\text{\rm (HP)}_i$ holds for all $i\in\{1,...,q\}$  then we just write $\text{\rm (HP)}$.
This hypothesis asks the friction parameter $\gamma$ to be  large enough, in order to limit the inertial effects, which induce oscillations (see Example REF). 
\if{In the particular mono-criteria case, this hypothesis can be assumed to be automatically satisfied. 
Indeed, since $\co \nabla f_i(u(t))$ is reduced to one point, the choice of the parameter $\alpha$ has absolutely no influence on the dynamic and can be supposed to be closed enough to $1$.}\fi
The hypothesis asks also implicitly the gradients $\nabla f_i$ to be Lipschitz over the trajectory (since $\gamma \in \R$).
Note that this last property holds whenever $u(\cdot)$ is bounded, since the gradients are Lipschitz continuous on bounded sets (see Corollary \ref{P:boundedness sufficient condition}).

As a direct consequence of the dissipative nature of the system, we obtain that the values $(f_i ( u(t))$ are bounded from above by $\max \{f_i(u_0);\E_i(t_0)\}$ :

\begin{corollary}{(Upper bound for the values)}\label{P:upper bound}
Let $u : [t_0,T[ \longrightarrow \H$ be a solution of $\IMOG$, such that  $\text{\rm (HP)}_i$ holds. Then, for all $i\in\{1,...,q\}$ and $t \geq t_0$, we have the following upper bounds :
\begin{equation}\label{E:upper bound}
f_i(u(t)) \leq \E_i(t_0) + (f_i(u_0) - \E_i(t_0)) e^{-\frac{\gamma}{m} (t-t_0)}.
\end{equation}
\end{corollary}

\begin{proof}
It is a trivial consequence of the monotonicity property of $\E_i$ obtained in Proposition \ref{P:dissipative property}. 
 Indeed, we obtain for all $t \in [t_0, +  \infty [$ :
\begin{equation}\label{ub1}
\frac{m}{\gamma} (f_i \circ u)'(t)  \leq \E_i(t_0) - (f_i \circ u)(t).
\end{equation}
The conclusion follows Gronwall's Lemma, applied to $t\mapsto (f_i \circ u)(t) - \E_i(t_0)$.
\end{proof}

This upper bound for the values has two interesting  consequences. The first one is immediate, and gives a useful sufficient condition for the trajectory $u(\cdot)$ to be bounded: 

{
\begin{corollary}\label{P:boundedness sufficient condition}
Suppose that there exists $i\in \{1,...,q\}$ such that $f_i$ 
 is coercive, and globally $L_i$-Lipschitz continuous, with $\gamma^2 > mL_i$. Then any trajectory of ${\rm(IMOG)}$ is bounded.
\end{corollary}
}
\noindent The second  consequence is that it tells us how to enforce the interesting property $f_i(u(\cdot)) \leq f_i(u_0)$. Indeed, we know that this dynamic is not a  descent method for the functions because of the inertial effects which can create damped oscillations. But at least, one can choose appropriately the initial velocity so that each point on the trajectory is better than the initial one.

\begin{corollary}\label{P:decreseness sufficient condition}
Suppose that $\text{\rm (HP)}_i$ holds for some $i\in\{1,...,q\}$.
For all $u_0 \in \H$, if $\dot u_0 \in \H$ is chosen to satisfy
\begin{equation}\label{E:decreseness sufficient condition}
\langle \nabla f_i (u_0) , \dot u_0 \rangle \leq - \gamma \Vert \dot u_0 \Vert^2,  
\end{equation}
then $f_i(u(t)) \leq f_i(u_0)$ for all $t \geq t_0$. 
In particular, for all $\lambda \in [0,\frac{1}{\gamma}]$, $\dot u_0=\lambda s(u_0)$ satisfies\footnote{Observe that the set of vectors satisfying this property recalls the notion of pseudo-gradient introduced by Miglierina \cite{Mig04}} (\ref{E:decreseness sufficient condition}).
\end{corollary}

\begin{proof}[Proof of Corollary \ref{P:decreseness sufficient condition}]
We see in Corollary \ref{P:upper bound} that the conclusion holds whenever $f_i(u_0)-\E_i(t_0) \geq 0$. This condition, once rewritten, is exactly (\ref{E:decreseness sufficient condition}). Now, consider $\dot u_0=\lambda s(u_0)$ for some $\lambda \in [0,\frac{1}{\gamma}]$. We recall that this steepest descent direction satisfies for all $i \in \{1,...,q\}$, see (\ref{E:common descent property}),
\begin{equation}
\Vert s(u_0) \Vert^2 + \langle \nabla f_i (u_0) , s(u_0) \rangle \leq 0.
\end{equation} 
So it follows easily that (\ref{E:decreseness sufficient condition}) holds for $\lambda s(u_0)$ .
\end{proof}

\if{
\begin{proposition}\label{P:boundedness sufficient condition}
Let $u(\cdot)$ be a trajectory of $(IMOG)$. Suppose that one of the objective functions is coercive, say $f_I$. Then the follwong propositions are equivalent :
\begin{itemize}
	\item[(i)]  $u(\cdot)$ is bounded ,
	\item[(ii)] For all $i \in \{1,...,q\}$, $Lip(\nabla f_i;u([0,+\infty[)) < + \infty$,
	\item[(iii)] $Lip(\nabla f_I;u([0,+\infty[)) < + \infty$.
\end{itemize}
In particular, the trajectory is bounded whenever one of the objective functions is coercive with a globally Lipschitz gradient.
\end{proposition}
}\fi

\subsection{Energy estimations}
We will now use the dissipative property of the system to deduce energy estimations for a global solution of ${\rm (IMOG)}$.

\begin{proposition}(Energy estimations)\label{P:energy estimations}
Let $u : [t_0,+\infty[ \longrightarrow \H$ be a bounded global solution of (IMOG) satisfying $\text{\rm (HP)}$. Then,
\begin{enumerate}
	\item[(i)] For all $i \in \{1,...,q\}$, \ $\E_i(t) \downarrow \E_i^\infty \in \R$ whenever $t \to + \infty$.
	\item[(ii)] $\dot u \in L^\infty(t_0,+\infty;\H) \cap L^2(t_0,+\infty;\H)$ and  $\ \lim\limits_{t \to +\infty} \Vert \dot u(t) \Vert = 0$.
	\item[(iii)] $\ddot u \in L^\infty(t_0,+\infty;\H)\cap  L^2(t_0,+\infty;\H)$ and $\ \underset{t \to +\infty}{\mbox{\em liminfess}} \ \Vert \ddot u(t) \Vert = 0$.
	\item[(iv)] For all $i\in\{1,...,q\}$, $(f_i \circ u)' \in L^\infty(t_0,+\infty; \R)$ and $\lim\limits_{t \to +\infty} (f_i \circ u)'(t) = 0$.
	\item[(v)] For all $i\in\{1,...,q\}$, $(f_i \circ u)\in L^\infty(t_0,+\infty; \R)$ and  $\lim\limits_{t \to +\infty} (f_i \circ u)(t) = \E_i^\infty$.
	\item[(vi)]  For all $i\in\{1,...,q\}$, there exists $\theta_i \in L^\infty(t_0,+\infty;\R)$ such that for all $t\in [t_0,T[$, $$m\ddot u(t) + \gamma \dot u(t) + \sum\limits_{i=1}^{q} \theta_i(t) \nabla f_i(u(t))=0 \text{ with } \theta(t) \in \S^q.$$
	In particular, it follows that $\sum\limits_{i=1}^{q} \theta_i(\cdot) (f_i \circ u)' \in L^1(t_0,+\infty;\H)$.
\end{enumerate}
\end{proposition}

\begin{proof}
We start by proving that $\dot u \in L^\infty(t_0,+\infty;\H)$, from which the other results will follow easily.
Take any $i \in \{1,...,q\}$, and define $c:= \inf\limits_{t \geq t_0} f_i(u(t)) - \E_I(t_0)$ and $M:=\max\limits_{i\in\{1,...,q\}} \ \sup\limits_{t\geq t_0} \Vert \nabla f_i(u(t)) \Vert$. 
Given that the gradients $ \nabla f_i$ are Lipschitz continuous on bounded sets, we deduce (using the mean value theorem) that the functions $f_i$ are bounded on bounded sets.
Since the trajectory is bounded, it follows that $M$ and $c$ are finite. 
In particular it implies that $m \ddot u + \gamma \dot u  \in L^\infty(t_0,+\infty;\H)$, since, according to (IMOG), we  have for a.e. $t \geq t_0$ that $-m \ddot u(t) -\gamma \dot u(t) \in \co \{\nabla f_i(u(t))\}$ which is bounded by $M$.

Using the monotonicity property of $\E_i$ (see Proposition \ref{P:dissipative property}), we have for all $t \geq t_0$:

\begin{equation}\label{E:ee1}
0 \geq \E_i(t) - \E_i(t_0) \geq m \Vert \dot u(t) \Vert^2 + \frac{m}{\gamma} (f_i\circ u)'(t) +c.
\end{equation}
Using Cauchy-Schwarz inequality and the definition of $M$, one has
\begin{equation}\label{E:ee2}
(f_i\circ u)'(t) = \langle \nabla f_i(u(t)), \dot u(t) \rangle \geq -\Vert \nabla f_i (u(t)) \Vert \Vert \dot u(t) \Vert \geq - M \Vert \dot u(t) \Vert.
\end{equation}
If we note $b=\frac{m}{\gamma}M$, we obtain
\begin{equation}\label{E:ee3}
0  \geq m \Vert \dot u(t) \Vert^2 - b \Vert \dot u(t) \Vert +c.
\end{equation}

\if{\color{blue}USELESS? Suppose in a first time that $c\leq 0$. Then it would follow, using (\ref{E:ee3}), that $  b \Vert \dot u(t) \Vert \geq a \Vert \dot u(t) \Vert^2 $ and so $\dot u \in L^\infty(0,+\infty;H)$. So we can assume that $c >0$.}\fi

\if{Suppose in a second time that $\alpha =0$. Then, according to $\alpha \ddot u + \gamma \dot u  \in L^\infty(0,+\infty;H)$, we would have again $\dot u \in L^\infty(0,+\infty;H)$. So we can assume that $\alpha >0$, which implies that $a >0$ in (\ref{E:ee3}).}\fi 

\noindent If we consider now the real polynomial $mX^2 - bX +c$ with $m>0$, we can see that it takes negative values on a  compact interval, independent of $t$. Since $\Vert \dot u(t) \Vert$ lies therein, we conclude that $\dot u \in L^\infty(t_0,+\infty;\H)$.

We can now derive the other properties, and we start with (i). The decreasing property of the energies $\E_i$ (see Proposition \ref{P:dissipative property}) ensures the existence of a limit $\E_i^\infty$, taking eventually the value $-\infty$. But now we can prove that for all $i\in\{1,...,q\}$, $\E_i^\infty\in \R$. Indeed, using the same inequality as in (\ref{E:ee2}),
\begin{equation}\label{E:ee4}
\E_i^\infty = \lim\limits_{t\to +\infty} \ \E_i(t) \geq \inf\limits_{t\geq t_0} f_i(u(t)) - \frac{m}{\gamma} M \Vert \dot u \Vert_{L^\infty(t_0,+\infty;\H)}  > - \infty.
\end{equation}

\if{It suffices now to prove that $\liminf\limits_{t\to +\infty} f_i(u(t))$ is bounded from below.
Let $t_n \to +\infty$ be a sequence such that $\liminf\limits_{t\to +\infty} f_i(u(t)) =\lim\limits_{n\to +\infty} f_i(u(t_n))$. 
Since $u(t_n)$ is bounded, it admits a weak limit point that we note $u^\infty$, and by taking eventually an appropriate subsequence we can assume that $u(t_n)$ converges to $u^\infty$. 
The objective functions being convex continuous, they are weakly lower-semicontinuous. 
So $f_i(u^\infty) \leq \liminf\limits_{t\to +\infty} f_i(u(t_n)) = \liminf\limits_{t\to +\infty} f_i(u(t))$ which gives us the desired lower bound.}\fi

We now prove (iii). 
Since $m \ddot u + \gamma \dot u$ and $\dot u$ lie in $L^\infty(t_0,+\infty;\H)$, we directly obtain from $m >0$ that $\ddot u  \in L^\infty(t_0,+\infty;\H)$. 
For the $L^2$ estimation, use Proposition \ref{P:dissipative property}  to obtain:
\begin{equation}
\frac{m^2}{\gamma} \displaystyle\int_{t_0}^{+\infty} \Vert \ddot u(t) \Vert^2\ dt \leq \int_{t_0}^{+\infty} -\frac{d}{dt} \E_i(t) \ dt \ = \E_i(t_0) - \E_i^\infty.
\end{equation}
It follows that $ \ddot u  \in L^2(t_0,+\infty;\H)$, and then, $\underset{t \to +\infty}{\mbox{ liminfess}} \Vert \ddot u(t) \Vert = 0$.

Let us now to prove (ii). Using exactly the same argument as for $\ddot u$, one obtains $ \dot u  \in L^2(t_0,+\infty\;H)$. Moreover, we know that $\dot u$ is Lipschitz continuous on $[t_0,+\infty[$ (since $\ddot u \in L^\infty(t_0,+\infty;\H)$), so it follows that $\lim\limits_{t \to +\infty} \Vert \dot u(t) \Vert = 0$.

We continue with items (iv) and (v). From Cauchy-Schwarz inequality, $\vert (f_i \circ u)'(t)\vert  \leq M \Vert \dot u(t) \Vert$ for all $t \geq t_0$. 
As a direct consequence of (ii), we deduce  $(f_i \circ u)'\in L^\infty(t_0,+\infty;\H)$ and $\lim\limits_{t \to +\infty} (f_i\circ u)'(t) = 0$. 
Then it follows directly from (i) that $\lim\limits_{t \to +\infty} (f_i\circ u)(t) = \E_i^\infty$, and $(f_i \circ u)\in L^\infty(t_0,+\infty;\H)$. 

We end the proof with item (vi). It is clear from the definition of (IMOG) that for all $t\geq t_0$, there exists $\theta(t)=(\theta_i(t))_{1\leq i \leq q} \in \S^q$ such that $m\ddot u(t) + \gamma \dot u(t) + \sum\limits_{i=1}^{q} \theta_i(t) \nabla f_i(u(t)) =0$. 
To get $\theta_i \in L^\infty(t_0,+\infty;\R)$, the whole point is to verify that it can be taken measurable. 
For this, we write $\theta(t)$ as a solution of the following optimality problem
\begin{equation}
\theta(t) \in \underset{\theta \in \S^q}{\mbox{ argmin}} \ j(t,\theta), \ \text{ where } j(t,\theta):=\Vert \sum\limits_{i=1}^{q} \theta_i \nabla f_i(u(t))  \Vert.
\end{equation}
Since $j$ is a Caratheodory integrand, we are guaranteed of the existence of a measurable selection $\theta : t \mapsto \theta(t) \in \underset{\theta \in \S^q}{\mbox{ argmin}} \ j(t,\theta)$ (see \cite[Propositions 14.6, 14.32 and 14.37]{RocWet}). Now we can write  
\begin{equation}
\sum\limits_{i=1}^{q} \theta_i(t) (f_i \circ u)'(t)  = {\sum\limits_{i=1}^{q} \theta_i(t) \langle \nabla f_i (u(t)) , \dot u(t) \rangle }= \langle -m \ddot u(t) - \gamma \dot u(t) , \dot u(t)\rangle
\end{equation}
where $\dot u, \ddot u \in L^2(t_0,+\infty;\H)$. So, using the Cauchy-Schwarz inequality and the measurability of $\theta_i$, we get directly that $\sum\limits_{i=1}^{q} \theta_i(\cdot) (f_i \circ u)' \in L^1(t_0,+\infty;\H)$.
\end{proof}

\subsection{Convergence of the trajectories of (IMOG)}

We present here the main result of this section. Under a convexity assumption, we show that the bounded trajectories of (IMOG) weakly converge to a solution.

\begin{theorem}\label{T:convergence result}
Suppose that the objective functions $f_i$ are convex. Then any bounded trajectory of (IMOG) $u : [t_0,+\infty[ \longrightarrow \H$ satisfying $\text{\rm (HP)}$ converges weakly to a weak Pareto optimum.
\end{theorem}

\noindent We sketch here the main points of the proof. 
The convergence essentially relies on Opial's Lemma that we recall below (note $\Omega [u(t)]$ the set of weak sequential cluster points of the trajectory) : 

\begin{lemma}\label{L:Opial} (Opial) \ Let   $S$ be
a non empty subset of $ \H$, and  $u: [t_0, +\infty [ \to \H$. Assume that
\begin{eqnarray*}
(i) && \Omega [u(t)] \subset S; \\
(ii) &&\mbox{for every }z\in S,\>  \lim\limits_{t \to  +\infty} 	\| u(t)- z 	\|  \mbox{ exists}.
\end{eqnarray*}
Then $u(t)$ weakly converges to some element
$u^{\infty}\in S$.
\end{lemma}

\noindent It is applied to the set 
$$S:=\{x\in \H \ | \ f_i(x) \leq \lim\limits_{t\to +\infty} f_i(u(t)) \ \text{ for all } i \in \{1,...q\} \ \},$$
for which (i) is easy to obtain. The key point to prove the F\'ejer property (ii) is that $h(t):=\frac{1}{2}\Vert u(t) - z \Vert^2$ satisfies a differential inequality. Indeed we have the following result from \cite[Lemma 4.2]{AttGouRed00} or \cite[Lemma 2.3]{AttMai11} :

\begin{lemma}\label{L:Alvarez}
Let $h \in \C^1(t_0,+\infty;\R)$ be a positive function satisfying $m \ddot h + \gamma \dot h \leq g$ where $m,\gamma >0 $ and $g \in L^1(t_0,+\infty;\R)$. Then $\lim\limits_{t \to  +\infty} h(t)$ exists.
\end{lemma}

\noindent Once obtained the weak convergence of the trajectory, the characterisation of its limit point as a weak Pareto point is a direct consequence of the demiclosedness property of $u \rightrightarrows \co\{\nabla f_i(u)\}$ (see for example \cite[Lemma 2.4]{AttGarGou15}) :

\begin{lemma}\label{L:demiclosed graph}
If $u_n^* \in \co \{\nabla f_i(u_n) \}$ with $\lim\limits_{n \to +\infty} u_n^* =0$ and $w-\lim\limits_{n \to +\infty} u_n =u_\infty$, then $0 \in \co \{\nabla f_i(u_\infty) \}$.
\end{lemma}

\begin{proof}[Proof of Theorem \ref{T:convergence result}]
Since $u$ is bounded, there exists some $t_n \to +\infty$ such that $u(t_n)$ converges weakly to some $u^\infty$. 
For all $i \in \{1,...q\}$, since $f_i$ is convex continuous it is in particular weakly semi-continuous. Hence, using Proposition \ref{P:energy estimations} we get
\begin{equation}\label{E:cr1}
f_i(u^\infty) \leq \liminf\limits_{n \to + \infty} f_i(u(t_n)) = \lim\limits_{t \to + \infty} f_i(u(t)).
\end{equation}
This proves that $\Omega[u(t)] \subset S \neq \emptyset$. 
To obtain convergence of the trajectory through Opial's Lemma, it remains to prove the Fejer property (ii). 
That is, given some $z \in S$, prove that $\lim\limits_{t \to  +\infty} 	\| u(t)- z 	\|$  exists. 

Define  $h(t):=\frac{1}{2}\Vert u(t) - z \Vert^2$ for all $t\geq 0$. 
Since $\dot u$ is absolutely continuous, then $h$ is twice differentiable for a.e. $t \in [0 , + \infty[$, and
\begin{eqnarray}
\dot h(t) &=&\langle \dot u(t), u(t) - z \rangle, \label{E:cr2} \\
\ddot h(t) &=& \langle \ddot u(t),u(t) - z \rangle + \Vert \dot u(t) \Vert^2. \label{E:cr3}
\end{eqnarray}
A linear combination of (\ref{E:cr2}) and (\ref{E:cr3}) gives
\begin{equation}\label{E:cr4}
m \ddot h(t) + \gamma \dot h(t) = m \Vert \dot u(t) \Vert^2 + \langle - m \ddot u(t) - \gamma \dot u(t) , z - u(t) \rangle.
\end{equation}
Let $\theta_i(t) \in \S^q$ be such that $- m \ddot u(t) - \gamma \dot u(t) = {\sum\limits_{i=1}^{q} \theta_i(t)\nabla f_i(u(t))}$, then we can rewrite
\begin{equation}\label{E:cr4.1}
m \ddot h(t) + \gamma \dot h(t) = m \Vert \dot u(t) \Vert^2 + \sum\limits_{i=1}^{q} \theta_i(t) \langle \nabla f_i(u(t)), z - u(t) \rangle.
\end{equation}
For any $i \in \{1,...,q\}$, we use the monotone property of $\E_i$ and $z \in S$ (recall that $\E_i^\infty=\lim\limits_{t \to + \infty} f_i(u(t))$) together with the convexity of $f_i$, to obtain for all $t \in [0,+\infty[$ :
\begin{eqnarray}
\E_i(t) &=& f_i(u(t)) + \frac{m}{\gamma} (f_i \circ u)'(t) +  m\Vert \dot u(t) \Vert^2 \label{E:cr5} \\
& \geq & \E_i^\infty \geq f_i(z) \geq f_i(u(t)) + \langle \nabla f_i (u(t)) , z - u(t) \rangle. \nonumber
\end{eqnarray}
Thus, it follows from (\ref{E:cr4.1}) and (\ref{E:cr5}) that 
\begin{equation}\label{E:cr7}
m \ddot h(t) + \gamma \dot h(t) \leq 2m \Vert \dot u(t) \Vert^2 +  \frac{m}{\gamma} \sum\limits_{i=1}^{q} \theta_i(t) (f_i \circ u)'(t),
\end{equation}
where the right member of (\ref{E:cr7}) lies in $L^1(t_0,+\infty;\H)$ (see Proposition \ref{P:energy estimations}).

\if{
We observe in the second member of (\ref{E:cr7}) that $\dot u \in  L^2(0,+\infty;H)$, see Proposition \ref{P:energy estimations}. Moreover, 
\begin{eqnarray}
\vert \sum\limits_{i=1}^{q} \theta_i(t) (f_i \circ u)'(t) \vert & = & \vert  \sum\limits_{i=1}^{q} \theta_i(t) \langle \nabla f_i(u(t)) , \dot u(t) \rangle \vert = \vert \langle -m \ddot u(t) - \gamma \dot u(t) , \dot u(t) \rangle  \vert \label{E:cr8} \\
& \leq & m \Vert \ddot u(t) \Vert \Vert \dot u(t) \Vert + \gamma \Vert \dot u(t) \Vert^2 \nonumber,
\end{eqnarray}
so we can also conclude by Proposition \ref{P:energy estimations} that $(\alpha - \varepsilon) \sum\limits_{i=1}^{q} \theta_i (f_i \circ u)' \in L^1(0,+\infty;\R)$. 
}\fi

Thus, hypothesis of Lemma \ref{L:Alvarez} is satisfied, and $\lim\limits_{t \to  +\infty} h(t)$ exists. 
It follows from Opial's Lemma that $u(t)$ weakly converges to some $u^\infty \in S$. It remains to prove that $u^\infty$ is a weak Pareto. 
In (IMOG), we have $-m \ddot u(t) - \gamma  \dot u(t) \in \co \{\nabla f_i(u(t)) \}$, where $w-\lim\limits_{t \to +\infty} u(t) = u_\infty$ and $\underset{t \to +\infty}{\text{liminfess}} \ \Vert m \ddot u(t) + \gamma  \dot u(t) \Vert=0$ (see Proposition \ref{P:energy estimations}).
Then we can apply Lemma \ref{L:demiclosed graph} to get $0 \in \co \{\nabla f_i(u^\infty) \}$. Following Proposition \ref{P:Fermat CNS}, this is equivalent for $u_\infty$ to be a weak Pareto point.
\end{proof}

\begin{remark}
If the objective functions are not convex, we still can say something on the limits points: each weak limit point of a bounded trajectory  of (IMOG) is a critical Pareto point (see Proposition \ref{P:energy estimations} and \ref{L:demiclosed graph}). 
\end{remark}

\section{Conclusion}

\if{
Here should appear
\begin{itemize}
	\item simple example(s), to validate and illustrate the previous results.
	\item example with complicate Pareto set, to show the exploration property of the dynamic (VS 1st order? VS scalarization?) cf paper of Miglioni
	\item state that convergence rate are the same that first-order on an easy example. Better rate should be available with friction of order $\frac{1}{t}$.
\end{itemize}
}\fi

{
We presented an inertial continuous dynamic for  multi-objective optimization, namely $\IMOG$.
We have shown the existence of global trajectories for $\IMOG$, and their asymptotic convergence in the convex case to weak Pareto points.
The general problem of the uniqueness of these trajectories remains open, as for the first-order dynamic (see \cite{AttGou14,AttGarGou15}).
Our study was motivated by the fact that inertial methods usually produce trajectories that converge more quickly than first-order methods.
It would now be interesting to study the rate of convergence of the trajectories of $\IMOG$ to weak Pareto points.
Given the recent results of \cite{CEG,SBC}, it seems natural to consider a modified version of $\IMOG$, allowing the viscosity parameter $\gamma$ to be time-dependent.
In particular, a dependence of the type $\gamma(t) =\frac{\alpha}{t}$ should open the road to FISTA-like algorithms for the resolution of multi-objective optimization problems.
Of course, these questions are out of the scope of this paper, and should be treated in a future work.

}

\appendix

\section{}

We give here the two integral forms of Gronwall's Lemma that we used in the proof of Theorem \ref{T:global existence}.
They can be found in Brezis's book \cite[Lemma A.4 \& Lemma A.5, pp. 156--157]{Br}.

\begin{lemma}[Gronwall-Bellman]\label{L:Gronwall-Bellman}
Let $t_0 \in \R$ and $T \in ]t_0,+\infty[$. 
Let $a \in [0,+ \infty[$, and $g\in L^1([0,T],\R)$ with $g(t)\geq 0$ for a.e. $t \in [0,T]$.
Let  $h \in C([0,T],\R)$ such that
\begin{equation}\label{E:Gronwall condition}
h(t) \leq a + \int_{t_0}^t g(s) h(s) \ ds \ \text{ for all } t \in [t_0,T].
\end{equation}
Then $h(t) \leq a e^{\int_{t_0}^t g(s) \ ds}$ for all $t \in [t_0,T]$.
\end{lemma}

\begin{lemma}\label{L:Gronwall 2}
Let $t_0 \in \R$ and $T \in ]t_0,+\infty[$. 
Let $a \in [0,+ \infty[$, and $g\in L^1([0,T],\R)$ with $g(t)\geq 0$ for a.e. $t \in [0,T]$.
Let  $h \in C([0,T],\R)$ such that
\begin{equation}\label{E:Gronwall condition 2}
\frac{1}{2} h^2(t) \leq \frac{a^2}{2} + \int_0^t g(s)h(s) \ ds \ \text{ for all } t \in [0,T],
\end{equation}
then $\vert h(t) \vert \leq a + \int_0^T g(s) \ ds \ \text{ for all } t \in [0,T].$
\end{lemma}

\end{document}